\documentclass[12pt,reqno]{amsart}
\usepackage{amsmath,amsfonts,amsthm,amssymb,color}

\usepackage[T1]{fontenc}
\usepackage{mathrsfs}
\usepackage{pdfsync}

\newcommand{\txx}{\textcolor}

\definecolor{dg}{rgb}{0, 0.5, 0}
\definecolor{dp}{rgb}{0.50, 0, 0.40}

\newcommand{\lp}{\left(}
\newcommand{\rp}{\right)}
\newcommand{\lc}{\left[}
\newcommand{\rc}{\right]}

\newcommand{\si}{\sigma}

\newcommand{\be}{\mathbf{E}}

\newcommand{\1}{{\bf 1}}

\newcommand{\beq}{\begin{equation}}
\newcommand{\eeq}{\end{equation}}
\newcommand{\beqs}{\begin{equation}}
\newcommand{\eeqs}{\end{equation}}
\newcommand{\bea}{\begin{eqnarray}}
\newcommand{\eea}{\end{eqnarray}}
\newcommand{\beas}{\begin{eqnarray*}}
\newcommand{\eeas}{\end{eqnarray*}}

\def\cC{{\mathcal C}}

\def\md{{\mathbb D}}

\def\mp{{\mathbb  P}}

\newcommand{\cac}{{\mathcal C}}

\newcommand{\cf}{{\mathcal F}}

\newcommand{\ch}{{\mathcal H}}

\newcommand{\D}{{\mathbb D}}

\newcommand{\R}{{\mathbb R}}

\newtheorem{theorem}{Theorem}[section]

\newtheorem{definition}[theorem]{Definition}

\newtheorem{lemma}[theorem]{Lemma}

\newtheorem{proposition}[theorem]{Proposition}
\theoremstyle{remark}
\newtheorem{remark}[theorem]{Remark}
\theoremstyle{remark}
\newtheorem{example}[theorem]{Example}
\theoremstyle{remark}

\newtheorem{foo}[theorem]{Remarks}

\title[Density for the solution of SDE with reflection driven by a fBm]{Existence of density for the solution of stochastic delay differential equations with reflection driven by a fractional Brownian motion}

\author[M. Besal\'u, D. M\'arquez-Carreras, C. Rovira]{M. Besal\'u, D. M\'arquez-Carreras, C. Rovira}

\address{Mireia Besal\'u, Dep. Gen\`etica, Microbiologia i Estad\'istica, Facultat de Biologia, Universitat de Barcelona. Diagonal, 645, 08028 Barcelona}
\email{mbesalu@ub.edu}
\address{David  M\'arquez-Carreras, Facultat de Matem\`atiques i Inform\`atica, Universitat de Barcelona. Gran Via de les Corts Catalanes, 585, 08007 Barcelona}
\email{davidmarquez@ub.edu}
\address{Carles Rovira, Facultat de Matem\`atiques i Inform\`atica, Universitat de Barcelona. Gran Via de les Corts Catalanes, 585, 08007 Barcelona}
\email{carles.rovira@ub.edu}

\thanks{ D. M\'arquez-Carreras and C. Rovira are supported by the grant PID2021-123733NB-I00 from SEIDI,  Ministerio de Economia y Competividad.}

\date{\today}

\begin{document}

\begin{abstract}%
In this note we prove the existence of a density for the law of the
solution for 1-dimensional stochastic delay differential  equations with normal reflection. The equations are driven by a fractional Brownian motion with Hurst parameter $H > 1/2$. The stochastic integral with respect to the fractional Brownian motion is a pathwise Riemann-Stieltjes integral.

\end{abstract}

\maketitle

\renewcommand{\theequation}{1.\arabic{equation}}
\setcounter{equation}{0}

\section{Introduction}

There are some models affected by some type of noise where the dynamics are related to propagation delay and some of them are naturally non-negative quantities. For instance, applications as rated and prices in internet model, concentrations of ions or proportions of a population that are infected   (see the references in \cite{K-W}). It is then natural to consider stochastic differential equations with delay and non-negativity constraints. In our paper \cite{B-R}, we initiate their study when the noise is  a fractional Brownian motion, obtaining the existence and uniqueness of solution. 

\medskip

More precisely, consider a stochastic delay differential equation with positivity constraints of the form:
\begin{eqnarray}\label{eqprincipal}
X_t&=&\eta_0+\int_0^t b(X_s)ds+\int_0^t \si(X_{s-r})dW_s^H+Y_t,\quad t\in(0,T_0],\nonumber\\
X_t&=& \eta_t,\qquad t\in[-r,0]. \label{eq:prin-frac}
\end{eqnarray} 
Here $r$ denotes a strictly positive time delay; $W^H$ is a fractional Brownian motion with Hurst parameter $H>\frac{1}{2}$ defined in a complete probability space $(\Omega,\cf,\mp)$;  $\eta:[-r,0]\rightarrow\R_+$ is a non negative smooth function, with $\R_+=[0,+\infty)$ and $Y$ is a vector-valued non-decreasing process which ensures that the non-negativity constraints on X are enforced.

\medskip

\noindent
Set
\begin{equation}
Z_t= \eta_0+\int_0^t b(X_s)ds+\int_0^t \si(X_{s-r})dW_s^H,\qquad t\in[0,T_0]. \label{eq:F}
\end{equation}
It is known that we have an explicit formula for the regulator term $Y$ in terms of $Z$:
\begin{equation*}
Y_t=\max_{s\in[0,t]} \left(Z_s\right)^-,\qquad t\in[0,T_0],
\end{equation*}
where $x^-= \max(0,-x)$. Then, using the called Skorohod's  mapping, the solution of \eqref{eq:prin-frac} satisfies
\[X_t=\qquad\begin{cases}
Z_t+Y_t,\qquad t\in[0,T_0],\\
\eta_t,\qquad\quad t\in[-r,0].
\end{cases}\]

\medskip

As $H > \frac12$, the integral with respect to $W^H$ can be defined as a pathwise Riemann-Stieltjes integral using the results given by Young \cite{Y}. Moreover, Z\"alhe \cite{Z} introduced a
generalized Stieltjes integral using the techniques of fractional calculus. In particular, she
obtained a formula for the Riemann-Stieltjes integral using fractional derivatives (see (2.2)
below). Using this formula, Nualart and R{\u{a}}{\c{s}}canu \cite{N-R} proved a general result on existence,
uniqueness and finite moments of the solution to a class of general differential equations. In \cite{B-R} we extended these results to equation (\ref{eq:prin-frac}).

\medskip

Nualart and Sausserau \cite{N-S} obtained the Malliavin differentiability and the existence of density for the solutions to equations considered in \cite{N-R}. Our aim is to extend these results to our equations and to study the existence of density of the solution to (\ref{eqprincipal}).

\medskip

As in all the papers following the methods presented in \cite{N-R}, we will obtain first deterministic results and then we will apply pathwise to the stochastic equations.

\medskip

Our main result states as follow:

\begin{theorem}\label{thprin}  Assume $b, \sigma \in \cac^2_b$ and $\eta \in \cac^{1-\alpha} (-r,0)$. Set $X$ the solution of (\ref{eqprincipal}) and assume that $\sigma >0$ on $\R^+$. Then, for any 
	$t_0 \in (0,T_0]$,  the law of $X_{t_0}$  restricted on $(0,+\infty)$ is absolutely continuous.
\end{theorem}

Notice that we obtain the regularity restricted on $(0,+\infty)$ since we are not able to prove that $P(X_{t_0}=0)=0$. It is the same situation that appears in the study of SPDEs \cite{DM-P-2}.
In our paper, the main difficult of the problem lies in the lack of properties of the Skorohod map with respect to the topologies used working with fractional Brownian motion.

\medskip

Stochastic systems with delay have been studied deeply (see, as a basic reference, \cite{M}), but the literature about stochastic differential equations with delay driven by a fractional Brownian
motion is scarce. There are references dealing with existence and uniqueness of solution \cite{FR, LT, NNT, TT},
existence and regularity of the density \cite{LT} and convergence when the delay goes to zero \cite{F-R}. 
There has been little work on stochastic differential equations with delay and non-negativity constraints driven by standard Brownian motion.
We can only refer the reader to  \cite{K},  dedicated to study numerical methods, and  \cite{K-W}
where the authors obtain sufficient conditions for existence and uniqueness of stationary solutions. 
The study of equations with fractional Brownian motion and non negative constraints is also weak.  The existence of solution is given in \cite{FRa} and the existence and uniqueness in \cite{FSa}. In \cite{DGHT} the authors deal with  existence and uniqueness using rough path techniques.

\medskip

On the other hand, the study of the regularity of the solution to stochastic differential equations with reflection is also scarce. We can refer the reader to the initial papers dealing with diffusions with reflection \cite{LM} and  \cite{LNS} and the work dedicated to stochastic partial differential equations \cite{DM-P-2} and \cite{NP}.

\medskip

The structure of the paper is as follows. In the next Section, we give some preliminaries about the Skorohod problem and fractional calculus. In Sections \ref{sec:comp}, \ref{sec:approx}, \ref{secdif} and \ref{seccon}  we deal with the deterministic case. Section 3 is devoted to get some comparison, existence and uniqueness theorems for deterministic equations.  In Section \ref{sec:approx} we obtain the convergence results to solve the deterministic Skorohod problem. The Fr\'echet differentiability of our deterministic equations is considered in Section \ref{secdif} and in Section \ref{seccon} we put together the deterministic results of the previous sections to obtain the approximations that we will consider pathwise in our stochastic case. Finally, Section \ref{secsto} is devoted to apply all the deterministic  results to the stochastic case proving our main theorem. 


\renewcommand{\theequation}{2.\arabic{equation}}
\setcounter{equation}{0}

\section{Preliminaries}

\subsection{Skorohod Problem}

Let  $$\cac_+(\R_+,\R):=\left\{x\in\cac(\R_+,\R): x(0)\in\R_+\right\}.$$
\vskip 5pt
\noindent
We can recall now the Skorokhod problem. 

\begin{definition} \label{def:skorokhod}
Given a path $z\in\cac_+(\R_+,\R)$, we say that a pair $(x,y)$ of functions in $\cac_+(\R_+,\R)$ solves the Skorokhod problem for $z$ with reflection if
\begin{enumerate}
\item $x_t=z_t+y_t$ for all $t\geq 0$ and $x_t\in\R_+$ for each $t\geq 0$,
\item $y(0)=0$ and $y$ is non-decreasing,
\item ${\displaystyle \int_0^t x_sdy_s=0}$ for all $t\geq 0$, so $y$ can increase only when $x$ is at zero.
\end{enumerate}
\end{definition}
\noindent
It is known that we have an explicit formula for $y$ in terms of $z$: 
\begin{equation*}
y_t=\max_{s\in[0,t]} \left(z_s\right)^-.
\end{equation*}

The path $z$ is called the reflector of $x$ and the path $y$ is called the regulator of $x$. We use the Skorokhod mapping for constraining a continuous real-valued function to be non-negative by means of reflection at the origin. We will apply it to each path of $z$ defined by \eqref{eq:F}.

\subsection{Fractional calculus}

Let $\alpha\in\lp 0,\frac12\rp$ and  $(a,b)\subseteq [-r,T]$. We denote by $W_1^{\alpha}(a,b)$ the space of measurable functions $f:[a,b]\rightarrow\R$ such that
\[\left\|f\right\|_{\alpha,1(a,b)}:=\sup_{u\in[a,b]} \Delta_t^\alpha(f)=\sup_{u\in[a,b]}\left(\left|f(u)\right|+\int_{a}^u\frac{\left|f(u)-f(v)\right|}{(u-v)^{\alpha+1}}dv\right)<\infty.\]

We also denote by $W_2^{1-\alpha}(0,T)$ the space of measurable functions $g:[0,T]\rightarrow\R$ such that
\[\left\|g\right\|_{1-\alpha,2}:=\sup_{0\leq u<v\leq T}\left(\frac{\left|g(v)-g(u)\right|}{|v-u|^{1-\alpha}}+\int_{u}^v\frac{\left|g(y)-g(u)\right|}{(y-u)^{2-\alpha}}dy\right)<\infty.\]
\noindent
Finally, we define the supremum norm for functions $f:[a,b]\rightarrow\R$ as
\[\|f\|_{\infty(a,b)}=\sup_{u\in[a,b]}|f(u)|,\]
the space of $\lambda$-H\"older continuous functions $f:[a,b]\rightarrow \R$ denoted by $\cac^\lambda (a,b)$ for any $0<\lambda\leq 1$ with the norm
\[\left\|f\right\|_{\lambda(a,b)}:=\left\|f\right\|_{\infty(a,b)}+\sup_{a\leq u<v\leq b}\frac{\left|f(v)-f(u)\right|}{(v-u)^{\lambda}}<\infty,\]
and for any integer $k\geq 1$ we denote by $\cac^k_b$ the class of real-valued functions on $\R$ which are $k$ times continuously differentiable with bounded partial derivatives up to the $k$-th order.
\vskip 5pt

\noindent
Clearly, for any $\varepsilon>0$,
\begin{equation}\label{inclusion}
C^{1-\alpha+\varepsilon}(0,T) \subset W^{1-\alpha}_2(0,T)
\subset C^{1-\alpha}(0,T).
\end{equation}
Moreover, as $\alpha \in (0, \frac12)$, 
$$C^{1-\alpha}(0,T) \subset W^{\alpha}_1(0,T).$$
If $f\in C^{\lambda }(a,b)$ and $g\in C^{\mu }(a,b)$ with $\lambda+\mu >1$, it is proved in \cite{Z} that the Riemman-Stieltjes integral $\int_{a}^{b}fdg$ exists and it can be expressed as 
\begin{equation}
\int_{a}^{b}fdg=(-1)^{\alpha }\int_{a}^{b}D_{a+}^{\alpha
}f(t)D_{b-}^{1-\alpha }g_{b-}(t)dt,  \label{eq:forpart}
\end{equation}
where $g_{b-}(t)=g(t)-g(b), $ $1-\mu <\alpha <\lambda $, and the fractional derivatives are defined as
\begin{eqnarray*}
D_{a+}^{\alpha }f(t) &=&\frac{1}{\Gamma (1-\alpha )}\left( \frac{f(t)}{%
(t-a)^{\alpha }}+\alpha \int_{a}^{t}\frac{f(t)-f(s)}{(t-s)^{\alpha +1}}%
ds\right) ,  \label{d2} \\[4pt]
D_{b-}^{\alpha }f(t) &=&\frac{(-1)^{\alpha }}{\Gamma (1-\alpha )}\left( 
\frac{f(t)}{(b-t)^{\alpha }}+\alpha \int_{t}^{b}\frac{f(t)-f(s)}{%
(s-t)^{\alpha +1}}ds\right).  \label{d3}
\end{eqnarray*}
We refer to \cite{N-R} and \cite{Z} and the references therein for a detailed account about this generalized integral and the fractional calculus.

\medskip

\noindent
Let $\Omega=C_0([0,T]; \R)$ be the Banach space of continuous functions, null at time 0, equipped with the supremum norm.
Let ${\rm P}$ be the unique probability measure on $\Omega$ such that the canonical process $\{W^H_t, t \in [0,T]\}$ is an $1$-dimensional fractional Brownian motion with Hurst parameter $H>\frac12$.

\medskip

\noindent
We denote by $\mathcal{E}$ the space of step functions on $[0,T]$ with values in $\R$. Let $\mathcal{H}$ be the Hilbert space defined as the closure of $\mathcal{E}$ with respect to the scalar product
$$\langle {\bf 1}_{[0, t]},{\bf 1}_{[0, s]} \rangle_{\mathcal{H}}= R_H(t, s),$$ 
where
$$R_H(t, s)=\int_0^{t \wedge s} K_H(t,r) K_H(s,r) dr,$$
and $K_H(t,s)$ is the square integrable kernel defined by
\begin{equation} \label{kh}
K_H(t,s)=c_H s^{1/2-H} \int_s^t (u-s)^{H-3/2} u^{H-1/2} du,\qquad\mbox{for}\;t>s,
\end{equation}
where $c_H=\sqrt{\frac{H(2H-1)}{\beta(2-2H, H-1/2)}}$ and $\beta$ denotes the Beta function. And for $t \leq s$, we set $K_H(t,s)=0$.

\medskip

\noindent
The mapping ${\bf 1}_{[0, t]} \rightarrow W^{H}_{t}$ can be extended to an isometry between $\mathcal{H}$ and the Gaussian space $\mathcal{H}_1$ associated to $W^H$. We denote this isometry by $\varphi \rightarrow W^H(\varphi)$.

\medskip

\noindent
Consider the operator $K^{\ast}_H$ from $\mathcal{E}$ to $L^2(0,T;\R)$
defined by
$$(K^{\ast}_H \varphi) (s)=\int_s^T \varphi(t) \partial_t K_H(t,s) dt.$$
From (\ref{kh}), we get
$$\partial_t K_H(t,s)=c_H \left(\frac{t}{s} \right)^{H-1/2} (t-s)^{H-3/2}.$$
Notice that
$$K^{\ast}_H({\bf 1}_{[0, t_1]},\ldots,{\bf 1}_{[0, t_m]} )
=(K_H(t_1,\cdot), \ldots, K_H(t_m,\cdot)).$$
For any $\varphi, \psi \in \mathcal{E}$, 
$$\langle \varphi, \psi  \rangle_{\mathcal{H}}=\langle K^{\ast}_H \varphi, K^{\ast}_H \psi \rangle_{L^2(0,T;\R^m)}={\rm E}(W^H(\varphi)W^H(\psi))$$
and $K^{\ast}_H$ provides an isometry between the Hilbert space $\mathcal{H}$ and a closed subspace of $L^2(0,T;\R)$.

\medskip

\noindent
Following \cite{N-S}, we consider the fractional version of the Cameron-Martin space $\mathcal{H}_H:=\mathcal{K}_H(L^2(0,T;\R))$, where for $h \in   L^2(0,T;\R)$,
$$(\mathcal{K}_H h)(t):=\int_0^t K_H(t,s) h_s ds.$$
We finally denote by $\mathcal{R}_H=\mathcal{K}_H \circ \mathcal{K}_H^{\ast}: \mathcal{H} \rightarrow\mathcal{H}_H$ the operator
$$\mathcal{R}_H \varphi=\int_0^{\cdot} K_H(\cdot, s) (\mathcal{K}^{\ast}_H h)(s) ds.$$
We remark that for any $\varphi \in \mathcal{H}$, $\mathcal{R}_H \varphi$ is H\"older continuous of order $H$. Therefore, for any $1-H<\alpha<1/2$,
$$\mathcal{H}_H \subset C^{H}(0,T;\R) \subset W^{1-\alpha}_2(0,T;\R).$$
Notice that $\mathcal{R}_H {\bf 1}_{[0,t]}=R_H(t, \cdot)$, and, as a consequence, $\mathcal{H}_H$ is the Reproducing Kernel Hilbert Space associated with the Gaussian process $W^H$. The injection
$\mathcal{R}_H: \mathcal{H} \rightarrow \Omega$ embeds $\mathcal{H}$ densely into $\Omega$ and for any $\varphi \in \Omega^{\ast} \subset  \mathcal{H}$,
$${\rm E} \left( e^{i W^H(\varphi)}\right)=
\exp \left( -\frac12 \Vert \varphi \Vert^2_{\mathcal{H}} \right).$$
As a consequence, $(\Omega, \mathcal{H}, {\rm P})$ is an abstract Wiener space in the sense of Gross.

 
\renewcommand{\theequation}{3.\arabic{equation}}
\setcounter{equation}{0}
\section{Comparison and existence and uniqueness theorems} \label{sec:comp}
\noindent
Let $0 < \alpha < \frac12$. Consider now the following equation
\begin{equation}
w_t = \eta_0+\zeta_t+\int_0^t b(w_s)ds,\quad t\in(0,T], \label{eq:th-exist}
\end{equation}
where $\zeta \in \cac^{1-\alpha} (0,T)$ and $\eta_0 >0$. 

\medskip

\noindent
In this section we will prove the existence and uniqueness of solution to equation \eqref{eq:th-exist} and a comparison result.
\begin{lemma} \label{lem:exist-unic}
Assume that $b$ is a Lipschitz function and that $\zeta \in \cac^{1-\alpha} (0,T)$. Then \eqref{eq:th-exist} has a unique solution $w \in \cac^{1-\alpha} (0,T)$.
\end{lemma}

\begin{proof}
	The proof is a simple case of the proof of Theorem 5.1 in \cite{N-S}.
\end{proof}

\begin{lemma} \label{lem:comp1bis2}
Let $b,\,\tilde{b}$ Lipschitz functions, such that $b\leq \tilde{b}$, $\zeta \in \cac^{1-\alpha} (0,T)$ and $t\in[0,T]$, then if
\begin{eqnarray*}
w_t&=&\eta_0+\zeta_t+\int_0^t b(w_s)ds,\\
\tilde{w}_t&=&\eta_0+\zeta_t+\int_0^t \tilde{b}(\tilde{w}_s)ds,
\end{eqnarray*}
we have that
$w_t\leq \tilde{w}_t$, for any $t \in [0,T]$.
\end{lemma}
\begin{proof}
For $k>0$, consider the function
\[\varphi_k(x)=\1_{\{x\geq 0\}}\int_0^x\int_0^y \rho_k(z)dzdy,\]
where
\[\rho_k(z)=\left\{
\begin{array}{ll}
     2kz, & \textrm{if} \;\; z\in\left[0,\frac{1}{k}\right],
  \\[5pt]
  2 \times \1_{\{z\geq \frac{1}{k}\}}, &\textrm{otherwise}.
\end{array}
\right.\] 
\noindent
From \cite[Theorem 5.1]{G}, we can check that $\varphi_k\in \cac^2_b$ and that their derivatives are bounded:
\[0\leq \varphi'_k(x)\leq 2x^+,\;\textrm{and}\; 0\leq  \varphi''_k(x)\leq 2\times\1_{\{x\geq 0\}},\]
where $x^+=\max(0,x)$. We also have that $\varphi_k(x)\uparrow (x^+)^2$ when $k\rightarrow 0$.

\medskip

\noindent
Then using this properties on $\varphi_k$ we have
\begin{eqnarray*}
\varphi_k(w_t-\tilde{w}_t)&=& \int_0^t \varphi'_k(w_s-\tilde{w}_s)(b(w_s)-\tilde{b}(\tilde{w}_s)) ds\\
&\leq& \int_0^t \varphi'_k(w_s-\tilde{w}_s)(\tilde{b}(w_s)-\tilde{b}(\tilde{w}_s)) ds\\
&\leq& C\int_0^t (w_s-\tilde{w}_s)^+|w_s-\tilde{w}_s| ds\\
&\leq& C\int_0^t \left((w_s-\tilde{w}_s)^+\right)^2 ds,
\end{eqnarray*}
where $C$ does not depend on $k$.
Now, letting $k\rightarrow 0$ the last inequality becomes
\[\left((w_t-\tilde{w}_t)^+\right)^2\leq C\int_0^t \left((w_s-\tilde{w}_s)^+\right)^2 ds,\]
and using the Gronwall's inequality we obtain that $\left((w_t-\tilde{w}_t)^+\right)^2=0$ and from here the desired result.
\end{proof}


\renewcommand{\theequation}{4.\arabic{equation}}
\setcounter{equation}{0}
\section{Convergence results} \label{sec:approx}
\noindent 
In this section, we consider the following equations
\begin{eqnarray} \label{eq:det}
x_t^{(\zeta)}&=&\eta_0+\zeta_t+ \int_0^t b(x_s^{(\zeta)}) ds+ y_t^{(\zeta)},\,\qquad t\in [0,T],
\end{eqnarray} 
where $y^{(\zeta)}$ is the regulator term of $x^{(\zeta)}$, 
and
\begin{eqnarray} \label{eq:eps}
x_t^{(\zeta),\varepsilon}&=&\eta_0+\zeta_t+ \int_0^t \lc b(x_s^{(\zeta),\varepsilon})+\frac{1}{\varepsilon}f(x_s^{{(\zeta),\varepsilon}}) \rc ds,\,\qquad t\in [0,T].
\end{eqnarray}

\noindent
We assume that $f$ is a $\cC^2_b$ not increasing function such that 
\[\begin{cases}
f(x)=0, & \mbox{if}\; x \ge 0, \\
0 < f(x) \le x^-, &\mbox{if}\; x<0,\end{cases}\]
with $x^-=\max(0,-x)$.

\medskip

\noindent
We will assume this definition of $f$ throughout the paper.

Following the same computations in \cite{B-R}, we can easily check that equation \eqref{eq:det} has a unique solution.
The purpose of this section is to prove that equation \eqref{eq:eps} has a unique solution and moreover, these solutions converge to the solution of equation \eqref{eq:det} when $\varepsilon$ goes to zero. The result states as follows:

\begin{theorem} \label{th:boconv}
Assume that $b$ is a Lipschitz function,  $\zeta \in \cac^{1-\alpha} (0,T)$.  Then the solutions $x^{(\zeta),\varepsilon}$ of \eqref{eq:eps} converge uniformly, as $\varepsilon\rightarrow 0$,  to $x^{(\zeta)}$,  the unique solution of the equation \eqref{eq:det}. Moreover, $x^{(\zeta)} \in \cac^{1-\alpha} (0,T).$
\end{theorem}

The proof of this theorem follows easily from the following three propositions in this section.

\begin{proposition} \label{th:sol-eps} 
Under the hypothesis of Theorem \ref{th:boconv}, equation \eqref{eq:eps} has a unique solution $x_t^{(\zeta),\varepsilon} \in  \cac^{1-\alpha} (0,T)$.
\end{proposition}

\begin{proof}
If we consider $\tilde{b}(x)= b(x)+\frac{1}{\varepsilon}f(x)$, since $f$ is Lipschitz we can apply Lemma \ref{lem:exist-unic} and we obtain the desired result.
\end{proof}

\begin{proposition} \label{th:conv-eps}
Under the hypothesis of Theorem \ref{th:boconv},  the solutions $x^{(\zeta),\varepsilon}$ of \eqref{eq:eps} converge uniformly to a continuous function $G$ on $[0,T]$ as $\varepsilon\rightarrow 0$.
\end{proposition}

\begin{proof}
Following the method used in \cite{DM-P}, we will organize the proof in several steps.

\smallskip
\textit{Step 1.} We consider the following equation for $t\in [0,T]$,
\begin{equation}
v_t=\eta_0+ \zeta_t+\int_0^t b(v_s)ds. \label{eq:v_eps}   
\end{equation}
We want to prove that $v_t\leq x_t^{(\zeta),\varepsilon}$ for all $\varepsilon>0$ and $t\in[0,T]$.

\smallskip
\noindent
Using Lemma \ref{lem:exist-unic}, equation \eqref{eq:v_eps} has a unique solution. Moreover, since $b\leq \tilde{b}$ ($\tilde{b}(x)=b(x)+\frac{1}{\varepsilon}f(x)$ where $\frac{1}{\varepsilon}f(x)\geq 0$), we can apply Lemma \ref{lem:comp1bis2} and we obtain  the desired result.
\vskip 5pt
\noindent
\textit{Step 2.} Now we consider 
\begin{equation}
u_t=\eta_0+ \zeta_t+\int_0^t b\lp u_s+\sup_{s'\leq s} u_{s'}^-\rp ds. \label{eq:U}
\end{equation}
We also define $\tilde{u}_t=u_t+\sup_{s'\leq t}u_{s'}^-$ and $\phi_t=\sup_{s'\leq t}u_{s'}^-$ so we can write $\tilde{u}_t= u_t+\phi_t$.

We observe that $\tilde{u}_t\geq 0$ and $\phi_t$ is an increasing function. Following the same ideas in the proof of Lemma \ref{lem:exist-unic} we can prove that equation \eqref{eq:U} has a unique solution.

We want to prove that $x_t^{(\zeta),\varepsilon}\leq \tilde{u}_t$ for all $\varepsilon>0$. We  write $\tilde{\theta}_t^\varepsilon=x_t^{(\zeta),\varepsilon}-\tilde{u}_t$. 

We can use again the function $\varphi_k$ defined in the proof of Lemma \ref{lem:comp1bis2} and we have
\begin{eqnarray} \label{eq:tilde_Z}
\varphi_k(\tilde{\theta}_t^\varepsilon) &=& \int_0^t \varphi'_k(\tilde{\theta}_s^\varepsilon) \lp b(x_s^{(\zeta),\varepsilon})+\frac{1}{\varepsilon}f(x_s^{(\zeta),\varepsilon})-b(\tilde{u}_s)\rp ds - \int_0^t \varphi'_k (\tilde{\theta}_s^\varepsilon)d\phi_s \nonumber\\
 &=& \int_0^t \varphi'_k(\tilde{\theta}_s^\varepsilon) \lp b(x_s^{(\zeta),\varepsilon})-b(\tilde{u}_s)\rp ds +\frac{1}{\varepsilon}\int_0^t \varphi'_k(\tilde{\theta}_s^\varepsilon)f(x_s^{(\zeta),\varepsilon})ds- \int_0^t \varphi'_k (\tilde{\theta}_s^\varepsilon)d\phi_s \nonumber\\
&=& T_1^\varepsilon(k)+T_2^\varepsilon(k)+T_3^\varepsilon(k).
\end{eqnarray}
Now we observe that,
\begin{equation*}
0\leq T_2^\varepsilon(k)\leq \frac{1}{\varepsilon}\int_0^t 2(\tilde{\theta}_s^\varepsilon)^+ (x_s^{(\zeta),\varepsilon})^- ds=0.
\end{equation*}
In fact, if $(\tilde{\theta}_s^\varepsilon)^+>0$, then $\tilde{\theta}_s^\varepsilon=x_s^{(\zeta),\varepsilon} -\tilde{u}_s>0$. So, since we have proved that $\tilde{u}_t\geq 0$, it is necessary that $x_s^{(\zeta),\varepsilon}> 0$ and then $(x_s^{(\zeta),\varepsilon})^-=0$.
\vskip 3pt

On the other hand, $T_3^\varepsilon(k)\leq 0$, since $\varphi'_k(\tilde{\theta}_s^\varepsilon)\geq 0$ and $\phi_t$ is an increasing process.
Then,
\begin{eqnarray*}
0&\leq& T_1^\varepsilon(k) \leq \int_0^t 2 (\tilde{\theta}_s^\varepsilon)^+ |b(x_s^{(\zeta),\varepsilon})-b(\tilde{u}_s)|ds\\
&\leq& C\int_0^t (\tilde{\theta}_s^\varepsilon)^+ |x_s^{(\zeta),\varepsilon}-\tilde{u}_s|ds= C \int_0^t (\tilde{\theta}_s^\varepsilon)^+ |\tilde{\theta}_s^\varepsilon| ds = C \int_0^t \lp(\tilde{\theta}_s^\varepsilon)^+\rp^2 ds.
\end{eqnarray*}
Using the previous computations of $T_1^\varepsilon(k)$, $T_2^\varepsilon(k)$ and $T_3^\varepsilon(k)$ on \eqref{eq:tilde_Z} and for $k\rightarrow 0$ we have
\begin{equation*}
\lp(\tilde{\theta}_t^\varepsilon)^+\rp^2 \leq C \int_0^t \lp(\tilde{\theta}_s^\varepsilon)^+\rp^2 ds.
\end{equation*}
Now we have just to apply Gronwall's lemma and we obtain that $( \tilde{\theta}_t^\varepsilon)^+=0$. So $x_t^{(\zeta),\varepsilon}\leq \tilde{u}_t$  for all $\varepsilon$.

\smallskip
\textit{Step 3.} Since we have seen that $v_t\leq x_t^{(\zeta),\varepsilon}\leq \tilde{u}_t$ and $|\tilde{u}_t|\leq 2 \sup_{s\leq t} |u_s|$ we can state
\begin{equation*}
|x_t^{(\zeta),\varepsilon}|\leq |v_t|+2 \sup_{s\leq t} |u_s|.
\end{equation*}
\vskip 3pt
\noindent
So,  
\begin{equation*}
\sup_{\varepsilon} |x_t^{(\zeta),\varepsilon}|\leq |v_t|+2 \sup_{s\leq t} |u_s|.
\end{equation*}

\smallskip
\textit{Step 4.} We can use Lemma \ref{lem:comp1bis2} to see that $x_t^{(\zeta),\varepsilon}$ is an increasing sequence when $\varepsilon$ decreases. Using that we have proved that it is bounded for all $\varepsilon>0$ we can define $G_t:=\lim_{\varepsilon\rightarrow 0} x_t^{(\zeta),\varepsilon}$. Since the sequence $\{x^{(\zeta),\varepsilon}\}_\varepsilon$ is increasing, the converge is uniform in $[0,T]$ (see \cite{NP}). That is, $\sup_t |x_t^{(\zeta),\varepsilon} -G_t|\rightarrow 0$ when $\varepsilon\rightarrow 0$, and so $G$ is a continuous function.
\end{proof}

\begin{proposition} \label{prop:G_sol}
If $G$ is the function appearing in Proposition \ref{th:conv-eps}, then $G_t$ is the solution of the equation \eqref{eq:det} and $G \in \cac^{1-\alpha} (0,T).$
\end{proposition}

\begin{proof}
We will prove that $G$ is a solution of our equation. So, by uniqueness we will obtain the desired result.
Then, we have to check that we have a pair $(G_t, y_t^{(\zeta)})\in \cac_+(\R_+,\R)$ that solves the Skorohod problem with reflection for $z^{(\zeta)}$ a path in $\cC_+(\R_+,\R)$. By Definition \ref{def:skorokhod} we have to prove the following facts
\begin{itemize}
\item[(i)] $G_t=z_t^{(\zeta)}+y_t^{(\zeta)}$ and $G_t\ge 0$, for each $t\ge 0$.
\item[(ii)] $y_0^{(\zeta)}=0$ and $y^{(\zeta)}$ is non-decreasing.
\item[(iii)] For all $t\in [0,T]$, $\int_0^t G_s dy^{(\zeta)}_s=0$.
\end{itemize}

\smallskip
\textit{Step 1}: Define, using monotone convergence theorem,
\begin{eqnarray}
y_t^{(\zeta)}&=&\lim_{\varepsilon\rightarrow 0}\frac{1}{\varepsilon}\int_0^t f\lp x_s^{(\zeta),\varepsilon}\rp ds,\nonumber\\
z_t^{(\zeta)}&=&\eta_0+\zeta_t+\lim_{\varepsilon\rightarrow 0} \int_0^t b(x_s^{(\zeta),\varepsilon})ds=\eta_0+\zeta_t+\int_0^t b(G_s)ds.\label{zzeta}
\end{eqnarray}
So, it only remains to check that $G_t\geq 0$, $\forall t\in [0,T]$.

\smallskip
Since
\[x_t^{(\zeta),\varepsilon}-\eta_0-\zeta_t-\int_0^t b(x_s^{(\zeta),\varepsilon})ds =\frac{1}{\varepsilon}\int_0^t f\lp x_s^{(\zeta),\varepsilon}\rp ds\]
and
\[\lim_{\varepsilon\rightarrow 0}\left[\varepsilon x_t^{(\zeta),\varepsilon}-\varepsilon \eta_0-\varepsilon \zeta_t-\varepsilon\int_0^t b(x_s^{(\zeta),\varepsilon})ds\right]= 0,\quad \forall\,t\in [0,T],\]
we have  
\[\lim_{\varepsilon\rightarrow 0}\int_0^t f \lp x_s^{(\zeta),\varepsilon}\rp ds= 0,\quad \forall\,t\in [0,T].\]
The monotone convergence theorem implies that 
$\int_0^t f\lp G_s\rp ds=0,$ $\forall\,t\in [0,T]$. Then, $G_s^-=0$ and $G_s\geq 0$.
\vskip 3pt
\noindent
\textit{Step 2}:
Obviously $y_0^{(\zeta)}=0$. Moreover,  $y^{(\zeta)}$ is non-decreasing since $f \ge 0.$

\smallskip
\textit{Step 3}:
With the definition of $y_t^{(\zeta)}$ in mind, we define the following two measures on $[0,T]$
\begin{displaymath}
\mu_\varepsilon(A)=\int_A\frac{1}{\varepsilon}f \lp x_s^{(\zeta),\varepsilon}\rp ds,\qquad \qquad 
\mu(A)=\int_A dy_s^{(\zeta)},\qquad \quad A\in \mathcal{B}([0,T]).\end{displaymath}
We have that $\mu_\varepsilon$ converges weakly to $\mu$, this means that
$$\lim_{\varepsilon\downarrow 0}\mu_\varepsilon([0,t]) =\mu([0,t]),\qquad \forall\,t\ge 0.$$
Observe that if $\varepsilon\le \varepsilon'$, then $x^{(\zeta),\varepsilon}_t\ge  x^{(\zeta),\varepsilon'}_t$ for all  $t \in [0,T]$, and $f(x^{(\zeta),\varepsilon}_t)\le  f(x^{(\zeta),\varepsilon'}_t),$ for all $t\in [0,T]$.
It implies that ${\rm supp}\ \mu_\varepsilon\subseteq \ {\rm supp} \mu_{\varepsilon'}$, and so, ${\rm supp}\ \mu\subseteq {\rm supp}\ \mu_{\varepsilon}$, for any $\varepsilon>0$.

\smallskip
So, if $s\in {\rm supp}\ \mu$, then $f(x^{(\zeta),\varepsilon}_s)\ge 0$ and it satisfies that $x^{(\zeta),\varepsilon}_s\le 0$. So, we get, for any $t \in [0,T]$,
\begin{displaymath}
\int_0^t x^{(\zeta),\varepsilon}_s dy_s^{(\zeta)}=\int_0^t x^{(\zeta),\varepsilon}_s d\mu (s) \le 0.
\end{displaymath}
Finally, by the monotone convergence theorem
\begin{displaymath}
\int_0^t G_s dy_s^{(\zeta)} \le 0.
\end{displaymath}
The condition $G\ge 0$ implies that 
\begin{displaymath}
\int_0^t G_s dy_s^{(\zeta)} =0.
\end{displaymath}
\medskip
\textit{Step 4}: To finish the proof let us check that $G \in \cac^{1-\alpha} (0,T).$ From expression (\ref{zzeta}) it is clear that $z^{\zeta}$ belongs to $\cac^{1-\alpha} (0,T)$ and, on the other hand, in the proof of  Proposition 4.2 in \cite{B-R} we have checked that  $\Vert G \Vert_{1-\alpha (0,T)} \le 2 \Vert z^{(\zeta)} \Vert_{1-\alpha (0,T)}$.
\end{proof}


\renewcommand{\theequation}{5.\arabic{equation}}
\setcounter{equation}{0}

\section{Differentiability}\label{secdif}

The aim of this section is to study the Fr\'echet differentiability in the direction $h\in W_1^{1-\alpha}(0,T)$ of the solution to equation
	\begin{eqnarray} 
&&x_t^g=\eta_0+\int_0^t b(x_s^g)ds+\int_0^t \sigma(\bar x^g_{s-r})dg_s, \qquad t\in[0,T],\nonumber\\
&&x_t^g=\eta_t,\qquad t\in[-r,0] \label{eq:detsig}
\end{eqnarray}
where $\bar x^g$ is Fr\'echet differentiable. We follow the ideas presented in \cite{N-S}. The main result of this section is the following:

\begin{theorem} \label{prop:derxf}
Let $\bar x$ be a map $g \mapsto \bar x^g$ from $W_2^{1-\alpha}(0,T)$ to $W_1^{\alpha}(-r,T-r),$ continuously Fr\'echet differentiable in the direction $h\in W_2^{1-\alpha}(0,T)$.
	Assume $b,\sigma\in \cac^2_b$ and set $x^g$ the solution to equation (\ref{eq:detsig}).
Then the mapping $$g \in W_2^{1-\alpha}(0,T)\mapsto x^g \in W_1^{\alpha}(0,T)$$ is Fr\'echet differentiable in the direction $h\in W_2^{1-\alpha}(0,T)$ and it holds that
	\begin{equation*}
	D_h x_t^g=   \int_0^t   b'(x_s^g)   D_h x^g_{s}  ds + \int_0^t \sigma (\bar x^g_{s-r})  dh_s + \int_0^t \sigma' (\bar x^g_{s-r}) D_h \bar x^g_{s-r}  dg_s.
	\end{equation*}
\end{theorem}
\begin{proof} 
It follows easily combining Propositions \ref{prop:sigma} and \ref{frechet}.	
\end{proof}

We begin studying the differentiability of the integral.

\begin{proposition} \label{prop:sigma}
Let $x$ be a map $g \mapsto \bar x^g$ from $W_2^{1-\alpha}(-r,T-r)$ to $W_1^{\alpha}(-r,T),$ continuously Fr\'echet differentiable in the direction $h\in W_2^{1-\alpha}(0,T)$. Let $\sigma\in\cac^1_b$. Then the mapping 
\[F: W_2^{1-\alpha}(0,T)\mapsto W_1^{\alpha}(0,T),\]
defined by
\[ F(g)_\cdot:=\int_0^\cdot \sigma (\bar x_{s-r}^g) dg_s\]
is Fr\'echet differentiable  in the direction $h\in W_2^{1-\alpha}(0,T)$ with directional derivative
\[D_h F(g)_t=\int_0^t \sigma (\bar x^g_{s-r})  dh_s + \int_0^t \sigma' (\bar x^g_{s-r}) D_h \bar x^g_{s-r}  dg_s.\]
\end{proposition}
\begin{proof}  Let us consider the mapping 
\[H: W_2^{1-\alpha}(0,T) \times  W_1^{\alpha}(-r,T-r) \mapsto W_1^{\alpha}(0,T),\]
defined by
\[ H(g,x):=\int_0^\cdot \sigma ( x_{s-r}) dg_s.\]
Following the same computations as in  Lemma 3 in \cite{N-S} we can check that $H$ is  
Fr\'echet differentiable with directional derivatives:
\begin{itemize}
\item for any $h\in W_2^{1-\alpha}(0,T)$
\[D_{1,h} H(g,x)_t=\int_0^t \sigma ( x_{s-r})  dh_s,\]
\item  for any $w\in W_1^{\alpha}(-r,T-r)$
\[D_{2,w} H(g,x)_t=\int_0^t \sigma' ( x_{s-r}) w_{s-r}  dg_s.\]
\end{itemize}

Then, since $F(g)=H(g,\bar x^g)$, using the chain rule we get that, for  $h\in W_2^{1-\alpha}(0,T)$
\begin{eqnarray*}
D_h F(g)_t &=& D_{1,h} H(g,\bar x^g)_t + D_{2,D_h \bar x^g} H(g,\bar x^g)_t \\
& =& \int_0^t \sigma (\bar x^g_{s-r})  dh_s + \int_0^t \sigma' (\bar x^g_{s-r}) D_h \bar x^g_{s-r}  dg_s.
\end{eqnarray*}
\end{proof}

We need a technical lemma before the last proposition.

\begin{lemma}\label{lema3}
Assume that $\eta_0 \in \R$ and that $b$ belongs to $C^2_b$ and that the mapping
\[\zeta: W_2^{1-\alpha}(0,T)\mapsto W_1^{\alpha}(0,T)\]
is Fr\'echet differentiable  in the direction $h\in W_2^{1-\alpha}(0,T)$.
Then the mapping 
$$F:W_2^{1-\alpha}(0,T) \times W_1^{\alpha}(0,T)\rightarrow W_1^{\alpha}(0,T)$$
defined by
\begin{equation*}
(g,x) \mapsto F(g,x):=x-\eta_0-\int_0^{\cdot} b(x_s) ds-\zeta^{g}_\cdot
\end{equation*}
is Fr\'echet differentiable in the direction $h \in W_2^{1-\alpha}(0,T)$. Moreover, for $h \in W_2^{1-\alpha}(0,T)$, $v\in W_1^{\alpha}(0,T)$ and $(g,x)\in W_2^{1-\alpha}(0,T) \times W_1^{\alpha}(0,T)$, the Fr\'echet derivatives with respect to $h$ and $v$ are given respectively by
\begin{align} \label{der1}
D_{1,h} F(g,x)_t&=-D_h \zeta_t^g, \\\label{der2}
D_{2,v} F(g,x)_t&=v_t- \int_0^t  b'(x_s) v_s ds.
\end{align}
\end{lemma}

\begin{proof}
	For $(h,x)$ and $(\tilde{h}, \tilde{x})$ in $W_2^{1-\alpha}(0,T) \times W_1^{\alpha}(0,T)$ we have
	$$
	F(h,x)_t-F(\tilde{h}, \tilde{x})_t=x_t-\tilde{x}_t-
	\int_0^t (b(x_s)-b(\tilde{x}_s)) ds -(\zeta_t^{h}-\zeta_t^{\tilde{h}}).
	$$
	Using  \cite[Proposition 2.2]{BR}, we get that
	\begin{equation*}
	\begin{split}
	\bigg\Vert x-\tilde{x}-
	\int_0^{\cdot} (b( x_s)-b(\tilde{x}_s)) ds \bigg\Vert_{\alpha,1} & \leq 
	c_{\alpha, T}\Vert x-\tilde{x} \Vert_{\alpha,1}.
	\end{split}
	\end{equation*}
	On the other hand, since $\zeta$ is differentiable, it will be continuous and  
	$	|\zeta^{h}-\zeta^{\tilde{h}}| $ can be controlled.
	Therefore, $F$ is continuous in both variables $(h,x)$. 
	
	We  next show the Fr\'echet differentiability. Let $v,w\in W_1^{\alpha}(0,T)$. By 
	\cite[Proposition 2.2]{BR}, we have that
	$$
	\Vert D_{2,v}F(h,x) - D_{2,w}F(h,x) \Vert_{\alpha,1} \leq c_{\alpha, T} \Vert v-w \Vert_{\alpha,1}.
	$$
	Thus, $D_{2,\cdot}F(h,x)$ is a bounded linear operator.
	Moreover,
	\begin{equation*}
	F(h, x+v)_t-F(h,x)_t-D_{2,v}F(h,x)_t 
	 =\int_0^t (b( x_s)-b(x_s+v_s)+ b'(x_s) v_s) ds.
	\end{equation*}
	By the mean value theorem and \cite[Proposition 2.2]{BR},
	$$\bigg\Vert\int_0^{\cdot} (b( x_s)-b(x_s+v_s)+ b'(x_s) v_s) ds \bigg\Vert_{\alpha, 1} \leq c_{\alpha, T}  \Vert v 		\Vert_{\alpha, 1}^2.$$
	This shows that (\ref{der2}) is the Fr\'echet derivative of $F(h,x)$ with respect to $x$. The Fr\'echet 		differentiability with respect to $h$ and the derivative is given easily by (\ref{der1}) from our hypothesis.
\end{proof}

\begin{proposition}\label{frechet}
	Assume the hypotheses of Lemma \textnormal{\ref{lema3}}. Let $x^g$ be the solution of 
	$$x_t^g = \eta_0+\zeta_t^g+\int_0^t b(x_s^g)ds,\quad t\in(0,T].$$
	Then the mapping $$g \in W_2^{1-\alpha}(0,T)\rightarrow x^g \in W_1^{\alpha}(0,T)$$ is Fr\'echet differentiable in 	the direction $h$ and  the derivative  is given by
	\begin{equation*}
	D_h x_t^g =D_h \zeta^g_t + \int_0^t  b'(x_s^g) D_h x_s^g ds.
	\end{equation*}
\end{proposition}

\begin{proof}
	Following the ideas of the proof of Proposition 4 in \cite{N-S} and the result of Lemma \ref{lema3} we have
\begin{eqnarray*}
D_{2,v} F(0,x^g)_t& =& v_t- \int_0^t  b'(x_s^g) v_s ds.\\
Dx^g&=&-D_{2,\cdot} F(0,x^g)^{-1} \circ D_1 F(0,x^g).
\end{eqnarray*}
	Then, for any $h \in W_2^{1-\alpha}(0,T)$, $-D_h x^g$ is the unique solution of the differential equation
	$$D_{1,h} F(0,x^g)_t= - D_{2,D_h x^g} F(0,x^g)_t,$$
	that is
	$$-D_h \zeta^g_t=-D_h x^g_t+ \int_0^t  b'(x_s^g) D_h x^g_s ds.$$
	So, we obtain the desired derivative.	
\end{proof}

We finish this section with an expression of the Fr\'echet derivative.

\begin{proposition} \label{frechet2}
	Assume the hypotheses of Theorem \ref{prop:derxf}. Assume that the derivative in the direction $h\in W_1^{1-\alpha}(0,T)$ is given by
	\begin{equation*}
	D_h \bar x_t^g = \int_0^t D_s \bar x_t^g dh_s,
	\end{equation*}
	 	with $D_s \bar x_u^g=0$ if $s>u$. Then the derivative of the solution of equation \eqref{eq:detsig} in the direction $h$ is given by
	\begin{equation*}
	D_h x_t^g = \int_0^t \Phi_t^{g}(s) dh_s,
	\end{equation*}
	 and $\Phi_t^g(s)$ defined as
\begin{equation}\label{phis}	
\Phi_t^{g}(s) =\begin{cases}
	\displaystyle\int_s^t b' (x_u^{g}) \Phi_u^g(s) du + \sigma (\bar x^g_{s-r}) + \int_s^t \sigma' (\bar x^g_{u-r}) D_s \bar x^g_{u-r} dg_u,&\qquad \mbox{if} \;s\leq t, \\[10pt]
	0,& \qquad \mbox{if} \;s> t.
	\end{cases}
	\end{equation}
\end{proposition}

\begin{proof}
	The proof follows similarly as the proof of \cite[Proposition 4]{N-S}. 
		Indeed, for $t \in [T]$
	\begin{eqnarray*}
		&& 	\int_0^t \Phi^g_t(s) dh_s =
		\int_0^t \Big(    \int_s^t b' (x_u^g) \Phi_u^g(s) du + \sigma (\bar x^g_{s-r}) + \int_s^t \sigma' (\bar x^g_{u-r}) D_s \bar x^g_{u-r}  dg_u \Big) dh_s \\
		&& =
		\int_0^t b' (x_u^g)  \Big( \int_0^u  \Phi_u^g(s)dh_s \Big)  du +  \int_0^t \sigma (\bar x^g_{s-r})    dh_s +
		\int_0^t \sigma' (\bar x^g_{u-r})  \Big( \int_0^{u-r}  D_s \bar x^g_{u-r}  dh_s \Big)  dg_u \\
		&& =
		\int_0^t b' (x_u^g)  D_h x_u^g du +  \int_0^t \sigma (\bar x^g_{s-r})    dh_s 
		+
		\int_0^t \sigma' (\bar x^g_{u-r})  D_h \bar x^g_{u-r}  dg_u,
	\end{eqnarray*}
	and $\Phi_t^{g}(s)=0$ if $s>t$. Notice that we have used that for $s>u-r$ we know that $D_s \bar x_{u-r}^g=0$
\end{proof}


\renewcommand{\theequation}{6.\arabic{equation}}
\setcounter{equation}{0}

\section{Definition of a sequence that converges to the deterministic solution}\label{seccon}

Our aim is to construct an approximated sequence in intervals of length $r$ that converges to the solution of the deterministic equation
\begin{equation*} 
\displaystyle x_t=\begin{cases}
\displaystyle\eta_0 + \int_0^t b (x_s)ds + \int_0^t \sigma(x_{s-r})dg_s + y_t, & \textrm{if} \;\; t\in\left[0,T\right],\\ 
\eta_t, & \textrm{if} \;\; t\in[-r,0].
\end{cases}  
\end{equation*}

We will not write the dependence on $g$ of the solution of this equation in order to simplify the notation. Let us recall that $b_\varepsilon(x)=b(x)+\frac1\varepsilon f(x)$ for $f$ defined as in Section \ref{sec:approx} and assume throughout this section that $b, \sigma \in \cac^2_b$, $\eta \in \cac^{1-\alpha} (-r,0)$ and $T=mr$. 

\medskip

We start with the interval $[-r,r]$. Since the map $t \to \int_0^t \sigma(\eta_{s-r})dg_s$ belongs to $ \cac^{1-\alpha} (0,r)$, we have seen in   Section \ref{sec:approx} that  the solutions to the equations
\begin{equation*} 
\displaystyle x^{\varepsilon}_t=\begin{cases}
\displaystyle\eta_0 + \int_0^t b_{\varepsilon} (x_s^{\varepsilon})ds + \int_0^t \sigma(\eta_{s-r})dg_s, & \textrm{if} \;\; t\in\left[0,r\right],
\\ \eta_u, & \textrm{if} \;\; t\in[-r,0],
\end{cases}  
\end{equation*}
converge uniformly to a continuous function $x$ belonging to  $ \cac^{1-\alpha} (-r,r)$ such that 
$$ \lim_{\varepsilon \to 0} \sup_{t \in [-r,r]} |x_t^{\varepsilon}-x_t |=0,$$
and that satisfies the reflected equation
\begin{equation*} 
\displaystyle x_t=\begin{cases}
\displaystyle\eta_0 + \int_0^t b (x_s)ds + \int_0^t \sigma(x_{s-r})dg_s + y_t, & \textrm{if} \;\; t\in\left[0,r\right],
\\ \eta_u, & \textrm{if} \;\; t\in[-r,0].
\end{cases}  
\end{equation*}

\medskip

Let us consider now the interval $[-r,2r]$ and the equation
\begin{equation*} 
x^{\varepsilon}_t=
\begin{cases}
\displaystyle x_t^{\varepsilon}, & \textrm{if} \;\; t\in\left[-r,r\right],
\\
\displaystyle x_r^{\varepsilon}  + \int_r^t b_{\varepsilon} (x^{\varepsilon}_s)ds + \int_r^t \sigma(x_{s-r})dg_s, & \textrm{if} \;\; t\in\left[r,2r\right],
\end{cases}
\end{equation*}
that can be written
$$ x^{\varepsilon}_t= \eta_0  + \int_0^t b_{\varepsilon} (x^{\varepsilon}_s)ds + \int_0^t \sigma(x_{s-r})dg_s,$$
for $t\in[0,2r]$. Following the same arguments, there exists a continuous function $x$ such that 
$$ \lim_{\varepsilon \to 0} \sup_{t \in [-r,2r]} |x_t^{\varepsilon}-x_t |=0,$$
and that satisfies the reflected equation
\begin{equation*} 
\displaystyle x_t=\begin{cases}
\displaystyle\eta_0 + \int_0^t b (x_s)ds + \int_0^t \sigma(x_{s-r})dg_s + y_t, & \textrm{if} \;\; t\in\left[0,2r\right],
\\ \eta_u, & \textrm{if} \;\; t\in[-r,0].
\end{cases}  
\end{equation*}

 Notice that we have used the same notation $x$ since the solution is the obvious extension of the function defined in $[-r,r].$ 

\medskip

Repeating the same argument $m$ times (recall that $T=mr$), we will obtain a  sequence 
\begin{equation*} 
x^{\varepsilon}_t=\begin{cases}
\displaystyle \eta_t, & \textrm{if} \;\; t\in\left[-r,0\right]\\
\displaystyle x_r^{\varepsilon}  + \int_r^t b_{\varepsilon} (x^{\varepsilon}_s)ds + \int_r^t \sigma(x_{s-r})dg_s, & \textrm{if} \;\; t\in\left[0,T\right],
\end{cases}
\end{equation*}
and a function $x$ belonging to  $ \cac^{1-\alpha} (0,T)$ satisfying
$$ \lim_{\varepsilon \to 0} \sup_{t \in [-r,T]} |x_t^{\varepsilon}-x_t |=0,$$
such that
$$x_t=\eta_0  + \int_0^t b (x_s)ds + \int_0^t \sigma(x_{s-r})dg_s + y_t,$$
for all $t\in [0,T]$.

\medskip

Now we have a sequence that converges to  $x$ uniformly in $(-r,T]$.
We want to deal now with its Fr\'echet differentiability.

\medskip

First, we deal again with the interval $[0,r]$. Clearly, from Theorem \ref{prop:derxf} yields that $x_t^{\varepsilon}$ are Fr\'echet differentiable in the directions $h \in W_2^{1-\alpha}(0,r)$  and that its derivative is, 
for $t\in[0,r]$
$$D_hx^{\varepsilon}_t = \int_0^t b'_{\varepsilon} (x_s^{\varepsilon}) D_hx_s^{\varepsilon} ds + \int_0^t \sigma(\eta_{s-r}) dh_s.$$
Moreover, from Proposition \ref{frechet2}, the derivative in the direction $h$ is given by
$$D_h x_t^{\varepsilon} = \int_0^t \Phi_t^{g,\varepsilon}(s) dh_s,$$
for $t \in [0,r]$ and $\Phi_t^\varepsilon(s)$ defined as
\[\Phi_t^{g,\varepsilon}(s) =\begin{cases}
\displaystyle\int_s^t b'_{\varepsilon} (x_u^{\varepsilon}) \Phi_u^{g,\varepsilon}(s) du + \sigma (\eta_{s-r}), &\qquad \mbox{if} \;s\leq t,\\[10pt]
0,& \qquad \mbox{if} \;s> t.
\end{cases}\]

\medskip

Let us consider a the general case, that is,  for $t \in [(l-1)r,lr]$.  We assume that $x_t$ is differentiable  in the direction $h$.   Then by Theorem \ref{prop:derxf} we clearly have that $x_t^{\varepsilon}$ are Fr\'echet differentiable in the direction $h$ and also we can yield its derivative.
\begin{eqnarray*}
	D_hx_t^{\varepsilon} &=&  D_h \int_0^t \sigma\left(x_{s-r}\right)dg_s + \int_0^t b'_\varepsilon \left(x_s^{\varepsilon}\right) D_hx_s^{\varepsilon} ds \\
& = &\int_0^t b'_{\varepsilon} \left(x_s^{\varepsilon}\right) D_hx_s^{\varepsilon} ds + \int_0^t \sigma \left(x_{s-r}\right) dh_s + \int_0^t \sigma' (x_{s-r}) D_h x_{s-r} dg_s.
\end{eqnarray*}
By Proposition \ref{frechet2}, this derivative in the direction $h$ can be written
$$D_h x_t^{\varepsilon} = \int_0^t \Phi_t^{g,\varepsilon}(s) dh_s,$$
for $t \in [0,lr]$ and $\Phi_t^\varepsilon(s)$ defined as
\[\Phi_t^{g,\varepsilon}(s) =\begin{cases}
\displaystyle\int_s^t b'_{\varepsilon} (x_u^{\varepsilon}) \Phi_u^{g,\varepsilon}(s) du + \sigma (x_{s-r}) + \int_s^t \sigma' (x_{u-r}) D_s x_{u-r} dg_u,&\qquad \mbox{if} \;s\leq t,\\[10pt]
0,& \qquad \mbox{if} \;s> t.
\end{cases}\]


\renewcommand{\theequation}{7.\arabic{equation}}
\setcounter{equation}{0}

\section{Stochastic case}\label{secsto}

In this section we apply the results obtained in the previous sections to the stochastic delay differential equation with positivity constraints \eqref{eq:prin-frac}. We can assume  without loss of generality that $T_0=mr$.

\medskip
Recall that $W^H=\{ W_t^H,\, t\in [0,T_0]\}$ is a one-dimensional fractional Brownian motion with Hurst parameter $H>\frac12$. That is, a centered Gaussian  process with covariance function
\[\be(W_t^{H} W_t^{H})=R_H(t,s)=\frac12\left( t^{2H}+s^{2H}-|t-s|^{2H}\right).\]
Fix $\alpha \in (1-H, \frac12)$.
As the trajectories of $W^H$ are $(1-\alpha+\epsilon)$-H\"older continuous for all $\epsilon<H+\alpha-1$,
by the first inclusion in (\ref{inclusion}), we can apply the framework of the previous sections.  In particular, under the assumptions of Theorem \ref{thprin}, there exist the solutions of equation \eqref{eq:prin-frac} and all the other stochastic equations that appear in this section.

We next proceed with the study of the Malliavin differentiability of the solution. 
We begin applying the deterministic results of Theorem \ref{prop:derxf} to our class of stochastic equations.

\begin{theorem} \label{d12loc}
	Let $\{\bar X_t, t \in [-r,kr]\} $ be a stochastic process in  $W_1^{\alpha}(-r,kr),$ such that  $\bar X$ is almost surely differentiable in the directions of the Cameron-Martin space. 
	Assume $b,\sigma\in \cac^2_b$ and set $\tilde X$ the solution to equation
\begin{eqnarray*} 
&&\tilde X_t=\eta_0+\int_0^t b(\tilde X_s)ds+\int_0^t \sigma(\bar X_{s-r})dW^H_s, \qquad t\in[0,r(k+1)],\nonumber\\
&&\tilde X_t=\eta_t,\qquad t\in[-r,0] .
\end{eqnarray*} 
Then, the solution $\tilde X$ is almost surely differentiable in the directions of the Cameron-Martin space. Moreover, for any $t>0$, the derivative satisfies 
	\begin{equation} \label{eq:der} D_s \tilde X_t=   \int_0^t   b'(\tilde X_u)   D_s \tilde X_u  du +  \sigma (\bar X_{s-r})   + \int_0^t \sigma' (\bar X_{u-r}) D_s \bar X_{u-r}  dW^H_u
	\end{equation}
	if $s\leq t$ and $D_s \tilde X_t=0$ if $s>t$.
\end{theorem}

\begin{proof}
	By Proposition \ref{frechet2}, the solution is Fr\'echet differentiable and for all $\varphi \in \mathcal{H}$  the Fr\'echet derivative
	$$D_{\mathcal{R}_H \varphi}\tilde X_t^i=\frac{d}{d\epsilon} \tilde X_t^i(\omega+\epsilon\mathcal{R}_H \varphi) \vert_{\epsilon=0}$$
	exists, which proves the first statement of the theorem.
	
	The derivative $D_{\mathcal{R}_H \varphi}\tilde X_t$  coincides with
	$\langle D \tilde X_t, \varphi \rangle_{\mathcal{H}},$
	where $D$ is the usual Malliavin derivative. Furthermore, for any $\varphi \in \mathcal{H}$, since we can write
	$$D_{\mathcal{R}_H \varphi} \bar X_t  =\langle D \bar X_t, \varphi \rangle_{\mathcal{H}},$$
	by Proposition \ref{frechet2}, we get,
	\begin{equation*} \begin{split}
	D_{\mathcal{R}_H \varphi} \tilde X_t &= \int_0^t \Phi_t^{W^H}(s) d(\mathcal{R}_H \varphi)(s)\\
	&= \int_0^t \Phi_t^{W^H}(s) \left( \int_0^s \partial_s K_H(s,u)(\mathcal{K}_H^{\ast} \varphi)(u) du\right) ds\\
	&=\sum_{j=1}^m \int_0^T (\mathcal{K}_H^{\ast} \Phi_t^{W^H})(s) (\mathcal{K}_H^{\ast} \varphi) (s) ds\\
	&=\langle \Phi_t^{W^H}, \varphi \rangle_{\mathcal{H}}
	\end{split}
	\end{equation*}
	and equation (\ref{eq:der}) follows from (\ref{phis}). This concludes the proof.
\end{proof}

\vskip 3pt
\noindent
Now, we can deal with the Malliavin derivability of the solution to the equation (\ref{eq:prin-frac}). First, we recall a technical lemma proved in \cite{Nu}.

\begin{lemma}\label{conve}
	Let $\{F_n,\,n\geq 1\}$ be a sequence of random variables in $\md^{1,p}$ that converges to F in $L^p(\Omega)$ and such that
	\[\sup_n\be\lc\|DF_n\|^p_\ch\rc<\infty.\]
	Then, $F$ belongs to $\md^{1,p}$, and the sequence of derivatives $\{DF_n,\,n\geq 1\}$ converges to $DF$ in the weak topology of $L^p(\Omega,\ch)$.
\end{lemma}

Using the lemma above we can state the following result:

\begin{proposition} \label{propd12}
For any $t \in (0,T_0]$, $X_t \in \D^{1,2}$.
\end{proposition}

\begin{proof}
Let us recall that we have for the deterministic case that $x_t= \lim_{\varepsilon \to 0 } x_t^\varepsilon$ and that $x_t^\varepsilon$ is an increasing sequence (Propositions \ref{th:conv-eps} and \ref{prop:G_sol}). So $X_t^\varepsilon$ converges almost surely to $X_t$. Since $X_t^\varepsilon$ is an increasing sequence and we know that $E(X_t^2)< \infty$, we get that $\{X_t^\varepsilon, \varepsilon>0\}$ converges to $X_t $ in $L^2$.

\medskip

We begin studying $t \in [0,r]$. 
Using Theorem \ref{d12loc}, since the initial condition $\eta$ is deterministic we get that, for any $\varepsilon >0$, $X_t^\varepsilon$ belongs to   $\md^{1,2}$ and its derivative satisfies, for $s < t$
$$ D_s X_t^\varepsilon = \int_s^t b'_{\varepsilon} (X_u^{\varepsilon}) D_s X_u^\varepsilon du + \sigma (\eta_{s-r}).$$
The solution of this equation can be written
\begin{eqnarray*} D_s X_t^\varepsilon &=& \sigma (\eta_{s-r}) \exp \left( \int_s^t b'_{\varepsilon} (X_u^{\varepsilon}) du \right) \\ &=&\sigma (\eta_{s-r}) \exp \left( \int_s^t b' (X_u^{\varepsilon}) du \right)  \exp \left( \frac{1}{\varepsilon}\int_s^t f' (X_u^{\varepsilon}) du \right).\end{eqnarray*} 
Using that 
\begin{equation} \label{xxxx} \be\lc\|DX_t^\varepsilon\|^2_\ch\rc \le C_H \int_0^t \be \left( \vert D_u X_t^\varepsilon \vert^2 \right) du, \end{equation}
since $f'<0$ we get easily that 
$$ \sup_\varepsilon \be\lc\|DX_t^\varepsilon\|^2_\ch\rc \le C_H \int_0^t \sigma^2 (\eta_{u-r}) \exp \left( 2 (t-u) \Vert b' \Vert_\infty \right) du < \infty.$$
Applying Lemma \ref{conve}, we get that for any $t \in [0,r], X_t \in \md^{1,2}$.

\medskip

We continue studying $t \in [0,2r]$. Using again Theorem \ref{d12loc}
we get that each $X_t^\varepsilon$ belongs to   $\md^{1,2}$ and its derivative satisfies, for $t \in [r,2r]$ and for $s < t$
$$ D_s X_t^\varepsilon = \int_s^t b'_{\varepsilon} (X_u^{\varepsilon}) D_s X_u^\varepsilon du + \sigma (X_{s-r}) + \int_s^t \sigma' (X_{u-r}) D_s X_{u-r} dW_u.$$
The solution of this equation can be written
\begin{eqnarray}
 \nonumber D_s X_t^\varepsilon &&=  \left( \sigma (X_{s-r}) + \int_s^t \sigma' (X_{u-r}) D_s X_{u-r} dW_u \right) \\ && \qquad \times\exp \left( \int_s^t b' (X_u^{\varepsilon}) du \Bigg)  \exp \Bigg( \frac{1}{\varepsilon}\int_s^t f' (X_u^{\varepsilon}) du \right). \label{expDXE}
	\end{eqnarray}
Using again (\ref{xxxx}),
since $f'<0$ we get easily that 
\begin{eqnarray*} &&\sup_\varepsilon \be \lc\|DX_t^\varepsilon\|^2_\ch\rc \le 2 C_H \int_0^t  \be  \left( \sigma^2 (X_{s-r}) + \left| \int_s^t \sigma' (X_{u-r}) D_s X_{u-r} dW_u \right|^2 \right)\\ && \hspace{5cm}
\times  \exp \left( 2 (t-u) \Vert b' \Vert_\infty \right) du < \infty.  \end{eqnarray*}
So, $X_t \in \md^{1,2}$.

\medskip

Iterating this procedure $m$ times, we get that 
 $X_t \in \md^{1,2}$ for any $t \in [0,T_0]$. Obviously, it implies that $Y_t \in \md^{1,2}$ for any $t \in [0,T_0]$ and the we can write,
for $s < t$
$$ D_s X_t = \int_s^t b' (X_u) D_s X_u du + \sigma (X_{s-r}) + \int_s^t \sigma' (X_{u-r}) D_s X_{u-r} dW_u + D_sY_t.$$
\end{proof}
\bigskip

\begin{remark}\label{feble}
	Notice that in the previous proof we have seen that
	$$ \sup_\varepsilon \int_0^{t_0} \be \left( \vert D_s X_{t_0}^\varepsilon \vert^2 \right) ds  < \infty.  $$
	So, there will exist a convergent sub-sequence that will converge in the weak topology of $L^2(\Omega;L^2([0,T]))$. Moreover, by uniqueness, the limit will be $DX_{t_0}.$
\end{remark}

Now we can give the proof of the main result of this paper.

\medskip

{\it Proof of Theorem \ref{thprin}.}
	We will prove that for any $a>0$, the restriction on $[a,+\infty)$ of the law of $X_{t_0}$ is absolutely continuous.
	We have seen in Proposition \ref{propd12} that $X_{t_0}$ belongs to $\D^{1,2}$. Then, using the classical result due to Bouleau and Hirsch (see \cite{BH} or \cite[Theorem 2.1.2]{Nu}) it is enough to check that on the set $\Omega_a=\{\omega, X_{t_0} (\omega) \ge a \}$
\begin{eqnarray*}
\int_0^{t_0}  (D_s X_{t_0} )^2 ds >0, \qquad\mbox{ a.s on }\Omega_a.
\end{eqnarray*}
Since
$$\int_0^{t_0}   \vert D_s X_{t_0} \vert ds < t_0^\frac12 \left(\int_0^{t_0}  (D_s X_{t_0} )^2 ds \right)^\frac12
$$
it suffices to check that
$$\int_0^{t_0}   \vert D_s X_{t_0} \vert ds > 0, \qquad \mbox{ a.s on }\Omega_a.$$
On the other hand, since $\sigma>0$, we have that $\sigma(X_{t_0-r})>0$ and we can find $t_1 (\omega) < t_0$ such that for $s\in[t_1,t_0]$
$$
\sigma (X_{s-r}) + \int_{s}^{t_0} \sigma' (X_{u-r}) D_{s} X_{u-r} dW_u > \frac12 \sigma(X_{t_0-r})>0.$$
From expression (\ref{expDXE}), we get that for any $s,\,  t_1(\omega)<s<t_0$ and for all $\varepsilon >0$
it holds that $D_s  X_{t_0}^\varepsilon >0$.

\medskip

Passing to the limit, we obtain that for any $s,\,  t_1(\omega)<s<t_0$  it holds that $D_s  X_{t_0} \ge 0$ and so, it will be enough to prove that
$$\int_{t_1}^{t_0}    D_s X_{t_0}  ds > 0, \qquad \mbox{ a.s on }\Omega_a.$$

\medskip

Fixed $a>0$ set $B_s=\{\omega, X_s(\omega) >\frac{a}{2} \}$ and $\tau_s = \inf \{t \ge s, X_t= \frac{a}{2} \}$. We will assume that our path $\omega$ is in $B_{t_0}.$ Then
\begin{equation}\label{igua0}
\int_{t_1}^{t_0}    D_s X_{t_0}  ds \geq \int_{t_1}^{t_0}    D_s X_{t_0} \1_{\{ \tau_s > t_0 \} }  ds.
\end{equation} 
From Remark \ref{feble} it yields that there exists a sub-sequence that will converge in the weak topology, such that, 
\begin{equation}\label{igua1}
\int_{t_1}^{t_0}    D_s X_{t_0} \1_{\{ \tau_s > t_0 \} }  ds =\lim_{\varepsilon \to 0}
\int_{t_1}^{t_0}    D_s X_{t_0}^\varepsilon  \1_{\{ \tau_s > t_0 \} }  ds,
\end{equation}
in the weak topology of $L^2(\Omega,L^2(0,T)).$
\medskip

Set $F_{s,t}^\varepsilon$ the solution to the integral equation, for $s<t$
$$F_{s,t}^\varepsilon=\sigma (X_{s-r}) + \int_{s}^t \sigma' (X_{u-r}) D_{s} X_{u-r} dW_u  + \int_s^t b'(X_u^\varepsilon) F_{s,u}^\varepsilon du.$$ 
Consider a  $s \in [t_1,t_0]$ with $\tau_s > t_0$. For all $u \in [s,t_0]$ it holds that $X_u \ge \frac{a}{2}.$

By the uniform convergence of $X^\varepsilon$ to $X$, there exists a  $\varepsilon_0(\omega)$ such that for all $\varepsilon < \varepsilon_0$ it holds that $X_u^\varepsilon > \frac{a}{4}$ for any $u \in [s,t_0]$. So, for all $\varepsilon < \varepsilon_0$ it holds that $f'(X_u^\varepsilon)=0$ for any $u \in [s,t_0]$. Then, for any $\varepsilon < \varepsilon_0(\omega)$ and for all $u \in [s,t_0]$ we have by uniqueness of the solution that
\begin{equation}\label{igua2}
F_{s,t_0}^\varepsilon= D_s X_{t_0}^\varepsilon.
\end{equation}

\medskip

For $s<t$, let us introduce
$$G_{s,t}=\sigma (X_{s-r}) + \int_{s}^t \sigma' (X_{u-r}) D_{s} X_{u-r} dW_u  + \int_s^t b'(X_u) G_{s,u} du.$$ 
So,
$$G_{s,t}= \left(\sigma (X_{s-r}) + \int_{s}^t \sigma' (X_{u-r}) D_{s} X_{u-r} dW_u \right) \exp\left( \int_s^t b'(X_u) du\right),$$
and it yields that $G_{s,t_0} >0$ for any $s \in [t_1,t_0]$.

\medskip

Notice that 
\begin{eqnarray*}
G_{s,t}- F_{s,t}^\varepsilon &= &\int_s^{t} b'(X_u) G_{s,u} du - \int_s^t b'(X_u^\varepsilon) F_{s,u}^\varepsilon du\\
 &= &\int_s^t( b'(X_u)-b'(X_u^\varepsilon)) G_{s,u} du + \int_s^t b'(X_u^\varepsilon)(G_{s,u}- F_{s,u}^\varepsilon) du
\end{eqnarray*}
Then,
$$G_{s,t_0} - F_{s,t_0}^\varepsilon = \left(\int_s^{t_0}( b'(X_u)-b'(X_u^\varepsilon)) G_{s,u} du \right) \exp \left( \int_s^{t_0} b'(X_u^\varepsilon)   du \right).$$
Using that $b'$ is differentiable with bounded derivative we have that
for $s \in [t_1,t_0]$
$$\vert G_{s,t_0} - F_{s,t_0}^\varepsilon \vert \le K \Vert X - X^\varepsilon \Vert_\infty  \left(\int_s^{t_0} \vert G_{s,u} \vert du \right).$$
and so $ F_{s,t_0}^\varepsilon$ converges a.s. to $G_{s,t_0}$.

\medskip

Using a bounded convergence result we get 
\begin{equation}\label{igua3}
\lim_{\varepsilon}
\int_{t_1}^{t_0}    F_{s,t_0}^\varepsilon  \1_{\{ \tau_s > t_0 \} }  ds= \int_{t_1}^{t_0}    G_{s,t_0} \1_{\{ \tau_s > t_0 \} }  ds>0, \qquad {\rm a.s.}
\end{equation}
Putting together (\ref{igua0}), (\ref{igua1}), (\ref{igua2}) and (\ref{igua3}) we obtain
$$\int_{0}^{t_0}    \left|D_s X_{t_0}\right|   ds= \int_{0}^{t_1}    \left|D_s X_{t_0}\right|   ds+\int_{t_1}^{t_0}    \left|D_s X_{t_0}\right|   ds
\ge \int_{t_1}^{t_0}    \left|D_s X_{t_0}\right|   ds\ge \int_{t_1}^{t_0}    D_s X_{t_0}   ds
>0.$$

\hfill$\square$


\begin{thebibliography}{99}


\bibitem{B-R} M. Besal\'u and C. Rovira: {\it Stochastic delay equations with non-negativity constraints driven by fractional Brownian motion.} {\rm Bernoulli} {\bf 18} (2012), {24-45}.

\bibitem{BR} M. Besal\'u and C. Rovira: {\it Stochastic Volterra equations driven by fractional Brownian motion with Hurst parameter $H>\frac12$.} {Stochastics and Dynamics} \textbf{12} (2012).

\bibitem{BH} B. Bouleau and F. Hirsch: {\it Propri\'et\'e d'absolute continui\'ee dans les espaces de Dirichlet el applications aux \'equations dif\' erentielles stochastiques.} {S\'eminaire de Probabilit\'es XX, Lectures Notes in Math} \textbf{1204} (1986).


\bibitem{DGHT} A. Deya, M. Gubinelli, M. Hofmanov\'a and S.Tindel: {\it One-dimensional reflected rough differential equations.} Stochastic Processes and their Applications
\textbf { 129}  (2019),  3261-3281,

\bibitem{DM-P} C. Donati-Martin and E. Pardoux: {\it White noise driven by SPDEs with reflection.} {\rm Probab. Theory Related Fields} {\bf 95} (1993), {1-24}.

\bibitem{DM-P-2} C. Donati-Martin and E. Pardoux: {\it EDPS R\'efl\'echies et Calcul de Maliiavin.} {\rm Bull. Sci. Math.} {\bf 121} (1997), {405-422}.


\bibitem{FR}
M. Ferrante and C. Rovira: \it Stochastic delay
differential equations driven by fractional Brownian motion with
Hurst parameter $H > 1/2$. \rm Bernoulli {\bf 12} (2006), 85-100.

\bibitem{F-R}  M. Ferrante and C. Rovira:
{\it Convergence of delay differential equations driven by fractional Brownian motion.} {\rm Journal of Evolution Equations} {\bf 10} (2010), {761--783}.

\bibitem{FRa} M. Ferrante and C. Rovira: {\it Stochastic differential equations with non-negativity constraints driven by fractional Brownian motion.} J. Evol. Equ. {\bf 13} (2013), 617--632.


\bibitem{FSa}  A. Falkowski, L.Slomi\'nsk: {\it  SDEs with constraints driven by semimartingales and processes with bounded p-variation.} Stochastic Processes and their Applications {\bf 127} (2017), 3536-3557,



\bibitem{G} I. Gy\"ongy:
{\it Existence and uniqueness results for semilinear stochastic partial differential equations.}
{\rm Stochastic Process. Appl.} {\bf 73} (1998), no. 2, 271-299.
\bibitem{K} H.J. Kusher: \it Numerical methods for controlled stochastic delay systems.  \rm Birkh\"auser, Boston (2008).

\bibitem{K-W}  M.S. Kinnally, J. Williams :
{\it On Existence and Uniqueness of Stationary
  Distributions for Stochastic Delay Differential Equations with Non-Negativity Constraints}. \rm Electronic Journal of Probability {\bf 15} (2010), 409-451.

\bibitem{LM} D. L\'epingle and C. Marois: {\it \'Equations diff\'erentielles stochastiques multivoques unidimensionnelles.}  S\'eminaire de Probabilit\'es, XXI, 520--533, Lecture Notes in Math., 1247, Springer, Berlin, 1987. 

\bibitem{LNS} D. L\'epingle, D.  Nualart and M. Sanz:  {\it
D\'erivation stochastique de diffusions r\'efl\'echies. }
Ann. Inst. H. Poincar\'e Probab. Statist.{\bf  25}  (1989), 283--305.

\bibitem{LT} J. Le\'on and  S. Tindel: {\it Malliavin calculus for fractional delay equations.}  J. Theoret. Probab. {\bf 25} (2012),  854--889.




\bibitem{M} S.-E. A. Mohammed: \it Stochastic differential systems with memory: theory, examples and
applications. In Stochastic Analysis and Related Topics VI (L. Decreusefond, J. Gjerde, B. Oksendal
and A.S. \"Ust\"unel, eds), \rm Birkh\"auser, Boston, 1-77 (1998).

\bibitem{NNT}
A. Neuenkirch, I. Nourdin, S. Tindel:
\it Delay equations driven by rough paths.
\rm Electron. J. Probab.  {\bf 13 } (2008),  2031--2068.

\bibitem{N-R} D. Nualart and A.  R{\u{a}}{\c{s}}canu: {\it Differential equations driven by fractional Brownian motion.} \rm Collect. Math. {\bf 53} (2002) 55-81.

\bibitem{NP} D. Nualart and E. Pardoux:
\it White noise driven quasilinear SPDEs with reflection.
\rm  Probab. Theory Related Fields {\bf 93} (1992) 77–89.

\bibitem{N-S} D. Nualart and B. Saussereau: {\it Malliavin calculus for stochastic differential equations driven by a fractional Brownian motion}. { Stochastic Process. Appl.} {\bf 119} (2009) 391--409.

\bibitem{Nu} D. Nualart: \emph{The Malliavin Calculus and Related
Topics.} Probability and its Applications. Springer-Verlag, 2nd

\bibitem{TT}
S. Tindel, I. Torrecilla: {\it Some differential systems driven by a fBm with Hurst parameter greater than 1/4.} Stochastic analysis and related topics, 169--202,
Springer Proc. Math. Stat., 22, Springer, Heidelberg, 2012.

\bibitem{Y} L.C. Young:  {\it An inequality of the H\"older type connected with Stieltjes integration.} Acta Math.
{\bf 67} (1936) 251--282. 

\bibitem{Z}    M. Z\"ahle:  {\it Integration with respect to fractal
functions and stochastic calculus  I.}  Prob. Theory Relat. Fields \textbf{111} (1998) 333--374.

\end{thebibliography}
\end{document}